\documentclass[12pt]{article}%
\usepackage{amsmath, amsthm}
\usepackage{amsfonts}
\usepackage{amssymb}
\usepackage{graphicx}
\usepackage{amsmath}%
\providecommand{\U}[1]{\protect\rule{.1in}{.1in}}
\newtheoremstyle{theorem}
{10pt}		
{10pt}
{\sl}
{\parindent}
{\bf}
{. }
{ }
{}
\theoremstyle{theorem}
\newtheorem{theorem}{Theorem}
\newtheorem{corollary}[theorem]{Corollary}
\newtheoremstyle{defi}
{10pt}		
{10pt}
{\rm}
{\parindent}
{\bf}
{. }
{ }
{}
\theoremstyle{defi}

\newtheorem{remark}{Remark}

\numberwithin{equation}{section}
\begin{document}

\title{A generalization of Ostrowski type inequality for mappings whose second
derivatives belong to L$_{1}\left(  a,b\right)  $ and applications}
\author{A. Qayyum$^{1}$, Ibrahima Faye$^{1}$, M. Shoaib$^{2}$ and M.A. Latif\\$^{1}$ Department of Fundamental and Applied Sciences, Universiti \\Teknologi Petronas, Malaysia.\\atherqayyum@gmail.com, ibrahima\_faye@petronas.com.my\\[2pt] $^{2}$ University of Hail, Department of Mathematics, PO BOX 2440,\\Kingdom of Saudi Arabia safridi@gmail.com}
\maketitle

\begin{abstract}
In this paper, we will improve and generalize inequality of Ostrowski type for
mappings whose second derivatives belong to L$_{1}\left(  a,b\right)  $ . Some
well known inequalities can be derived as special cases. In addition,
perturbed mid-point inequality and perturbed trapezoid inequality are also
obtained. The obtained inequalities have immediate applications in numerical
integration where new estimates are obtained for the remainder term of the
trapezoid and midpoint formula. Applications to some special means are also investigated.



\end{abstract}

\section{Introduction}

The development of inequalities has been established with the publication of
the books by Hardy, Littlewood and Polya \cite{1} in 1934, Beckenbach and
Bellman \cite{2} in 1961 and by Mitrinovi%
\'{}%
c, Pecari%
\'{}%
c and Fink \cite{3}-\cite{4} in 1991. The publication of later has resulted to
bring forward some new integral inequalities involving functions with bounded
derivatives that measure bounds on the deviation of functional value from its
mean value namely, Ostrowski inequality \cite{5}. This Ostrowski type
inequality has powerful applications in numerical integration, probability and
optimization theory, stochastic, statistics, information and integral operator theory.

During the last few years, many researchers focused their attention on the
study and generalizations \cite{6}, \cite{7}, \cite{8}, \cite{13}, \cite{14}
\ and \cite{15} of the Ostrowski inequality. In \cite{9}, Rafiq et.al proved
their results for second differentiable mappings by using Ostrowski-Gr\"{u}ss
type inequality. We prove our result for L$_{1}$ norm by using only Ostrowski
type inequality.

Ostrowski \cite{5} proved the classical integral inequality which is stated as:

\begin{theorem}
Let \ \ $f\ $: $I\subseteq$ $%
\mathbb{R}
\rightarrow%
\mathbb{R}
$ \ be\ a differentiable mapping on $I^{\circ}$($I^{\circ}$ is the interior of
$I$) and let $a,b\in I^{\circ}$with $a<b.$ If \ \ $f^{\text{ }\prime}\ $:
$\left(  a,b\right)  \rightarrow%
\mathbb{R}
$ \ is bounded on $\left(  a,b\right)  $ i.e. $\left\Vert f^{\prime
}\right\Vert _{\infty}=\underset{t\in\left[  a,b\right]  }{\sup}\left\vert
f^{\prime\prime}\left(  t\right)  \right\vert <\infty,$ then%
\begin{equation}
\left\vert f(x)-\frac{1}{b-a}\int\limits_{a}^{b}f(t)dt\right\vert \leq\left[
\frac{1}{4}+\frac{\left(  x-\frac{a+b}{2}\right)  ^{2}}{\left(  b-a\right)
^{2}}\right]  \left(  b-a\right)  \left\Vert f^{\prime}\right\Vert _{\infty}
\label{1}%
\end{equation}
for all $x\in\left[  a,b\right]  $. The\ constant$\ \frac{1}{4}\;$is sharp in
the sense that it can not be replaced by a smaller one.
\end{theorem}

In 1976, Milovanovi\'{c} et al. proved a generalization of Ostrowski's
inequality for twice differentiable mappings \cite{10}.

\begin{theorem}
Let \ \ $f\ $: $\left[  a,b\right]  \rightarrow%
\mathbb{R}
$ \ be\ a twice differentiable such that $f^{\text{ }\prime\prime}:$ $\left(
a,b\right)  \rightarrow%
\mathbb{R}
$ is bounded on $\left(  a,b\right)  ,$ i.e $\left\Vert f^{\prime\prime
}\right\Vert _{\infty}=\underset{t\in\left[  a,b\right]  }{\sup}\left\vert
f^{\prime\prime}\left(  t\right)  \right\vert <\infty$. Then the inequality
holds:%
\begin{align*}
&  \left\vert \frac{1}{2}\left[  f(x)+\frac{\left(  x-a\right)  f\left(
a\right)  +\left(  b-x\right)  f\left(  b\right)  }{b-a}\right]  -\frac
{1}{b-a}\int\limits_{a}^{b}f(t)dt\right\vert \\
&  \leq\frac{\left\Vert f^{\prime}\right\Vert _{\infty}}{4}\left(  b-a\right)
^{2}\left[  \frac{1}{12}+\frac{\left(  x-\frac{a+b}{2}\right)  ^{2}}{\left(
b-a\right)  ^{2}}\right]
\end{align*}
for all $x\in\left[  a,b\right]  .$
\end{theorem}

Dragomir and Wang \cite{7} proved (\ref{1}) for $f^{\text{ }\prime}\in
L_{1}\left[  a,b\right]  $ $,$ as follows:

\begin{theorem}
Let \ \ $f\ $:$I\subseteq$ $%
\mathbb{R}
\rightarrow%
\mathbb{R}
$ \ be\ a differentiable mapping in $I^{\circ}$ and $a,b\in I^{\circ}$with
$a<b.$ If $f^{\text{ }\prime}\in L_{1}\left[  a,b\right]  $, then the
inequality holds%
\begin{equation}
\left\vert \text{ }f(x)-\frac{1}{b-a}\int\limits_{a}^{b}f(t)dt\right\vert
\leq\left[  \frac{1}{2}+\frac{\left\vert x-\frac{a+b}{2}\right\vert }%
{b-a}\right]  \left\Vert f^{\prime}\right\Vert _{1} \label{2}%
\end{equation}
for all $x\in\left[  a,b\right]  $.
\end{theorem}

J. Roumeliotis \cite{11}, presented product inequalities and weighted
quadrature. The weighted inequlity was also obtained in Lebesgue spaces
involving first derivative of the function, which is given by%
\begin{align}
&  \left\vert \text{ }\frac{1}{b-a}\int\limits_{a}^{b}w\left(  t\right)
f(t)dt-m\left(  a,b\right)  f\left(  x\right)  \right\vert \label{3}\\
&  \leq\frac{1}{2}\left[  m\left(  a,b\right)  +\left\vert m\left(
a,x\right)  -m\left(  x,b\right)  \right\vert \right]  \left\Vert
f^{\prime\prime}\right\Vert _{1}\nonumber
\end{align}
Barnett et,al, \cite{12} proved an inequality of Ostrowski type for twice
differentiable mappings which is in terms of the $L_{1}$ norm of the second
derivative $f^{\prime\prime}$. The main result from \cite{12} is pointed out
in the theorem below:

\begin{theorem}
Let \ \ $f\ $: $\left[  a,b\right]  \rightarrow%
\mathbb{R}
$ \ be\ continuous on $\left[  a,b\right]  $ and twice differentiable in
$\left(  a,b\right)  $ and $f^{\text{ }\prime\prime}\in L_{1}\left(
a,b\right)  .$ Then the inequality holds%
\begin{align}
&  \left\vert \text{ }f(x)-\frac{1}{b-a}\int\limits_{a}^{b}f(t)dt-\left(
x-\frac{a+b}{2}\right)  f^{\text{ }\prime}\left(  x\right)  \right\vert
\label{4}\\
&  \leq\frac{1}{2\left(  b-a\right)  }\left(  \left\vert x-\frac{a+b}%
{2}\right\vert +\frac{1}{2}\left(  b-a\right)  \right)  ^{2}\left\Vert
f^{\prime\prime}\right\Vert _{1}\nonumber\\
&  \leq\frac{b-a}{2}\left\Vert f^{\prime\prime}\right\Vert _{1}\nonumber
\end{align}
for all $x\in\left[  a,b\right]  $.
\end{theorem}

Motivated and inspired by the work of the Barnett et.al \cite{12}, we will
establish a new generalized inequality. Some other interesting inequalities
are also presented as special cases. In the last, we will present applications
for some special means and in numerical integration.

\section{Main Results}

\begin{theorem}
\textit{Let \ }$\mathit{\ }f:[a,b]\rightarrow%
\mathbb{R}
$ \textit{be continuous on }$[a,b]$ and twice differentiable on $\left(
a,b\right)  $ with second derivative and $f^{\text{ }\prime\prime}:\left(
a,b\right)  \rightarrow%
\mathbb{R}
,$ for all $x\in\left[  a+h\frac{b-a}{2},b-h\frac{b-a}{2}\right]  ,$ \ it
follows that%
\begin{align}
&  \left\vert
\begin{array}
[c]{c}%
\left(  1-h\right)  \text{ }f(x)-\left(  1-h\right)  \left(  x-\frac{a+b}%
{2}\right)  f^{\text{ }\prime}\left(  x\right)  +\frac{h}{2}\left(  f\left(
a\right)  +f\left(  b\right)  \right) \\
\\
-\frac{h^{2}\left(  b-a\right)  }{8}\left(  f^{\text{ }\prime}\left(
b\right)  -f^{\text{ }\prime}\left(  a\right)  \right)  -\frac{1}{b-a}%
\int\limits_{a}^{b}f(t)dt
\end{array}
\right\vert \label{5}\\
&  \leq\frac{1}{2\left(  b-a\right)  }\left[  \frac{1}{2}\left(  b-a\right)
\left(  1-h\right)  +\left\vert x-\frac{a+b}{2}\right\vert \right]
^{2}\left\Vert f^{\text{ }\prime\prime}\right\Vert _{1}\nonumber\\
&  \leq\frac{\left(  b-a\right)  }{2}\left(  1-\frac{h}{2}\right)
^{2}\left\Vert f^{\text{ }\prime\prime}\right\Vert _{1}\nonumber
\end{align}
for all $x\in\left[  a+h\frac{b-a}{2},b-h\frac{b-a}{2}\right]  $ and
$h\in\left[  0,1\right]  .$
\end{theorem}

\begin{proof}
Let us define the mapping $K:[a,b]^{2}\rightarrow%
\mathbb{R}
$ by%
\[
K\left(  x,t\right)  =\left\{
\begin{array}
[c]{c}%
\frac{1}{2}\left[  t-\left(  a+h\frac{b-a}{2}\right)  \right]  ^{2}\text{, if
}t\in\left[  a,x\right] \\
\\
\frac{1}{2}\left[  t-\left(  b-h\frac{b-a}{2}\right)  \right]  ^{2}\text{, if
}t\in\left(  x,b\right]
\end{array}
\right.
\]
The proof uses the following identity:%
\begin{align}
\int\limits_{a}^{b}f(t)dt  &  =\left(  b-a\right)  \left(  1-h\right)
f(x)-\left(  b-a\right)  \left(  1-h\right)  \left(  x-\frac{a+b}{2}\right)
f^{\text{ }\prime}\left(  x\right) \label{6}\\
&  +h\frac{b-a}{2}\left(  f\left(  a\right)  +f\left(  b\right)  \right)
-\frac{h^{2}\left(  b-a\right)  ^{2}}{8}\left(  f^{\text{ }\prime}\left(
b\right)  -f^{\text{ }\prime}\left(  a\right)  \right)  +\int\limits_{a}%
^{b}K(x,t)f^{\text{ }\prime\prime}\left(  t\right)  dt\nonumber
\end{align}
for all $x\in\left[  a+h\frac{b-a}{2},b-h\frac{b-a}{2}\right]  .$
Using the identity (\ref{6})$,$ we have%
\begin{align}
&  \left\vert
\begin{array}
[c]{c}%
\left(  1-h\right)  f(x)-\frac{1}{\left(  b-a\right)  }\int\limits_{a}%
^{b}f(t)dt+\frac{h}{2\left(  1-h\right)  }\left(  f\left(  a\right)  +f\left(
b\right)  \right) \\
\\
-\frac{h^{2}\left(  b-a\right)  }{8}\left(  f^{\text{ }\prime}\left(
b\right)  -f^{\text{ }\prime}\left(  a\right)  \right)  -\left(  1-h\right)
\left(  x-\frac{a+b}{2}\right)  f^{\text{ }\prime}\left(  x\right)
\end{array}
\right\vert \nonumber\\
& \nonumber\\
&  =\frac{1}{\left(  b-a\right)  }\left\vert \int\limits_{a}^{b}%
K(x,t)f^{\text{ }\prime\prime}\left(  t\right)  dt\right\vert \nonumber\\
&  =\frac{1}{\left(  b-a\right)  }\left\vert \int\limits_{a}^{x}\frac{\left[
t-\left(  a+h\frac{b-a}{2}\right)  \right]  ^{2}}{2}f^{\text{ }\prime\prime
}\left(  t\right)  dt+\int\limits_{x}^{b}\frac{\left[  t-\left(  b-h\frac
{b-a}{2}\right)  \right]  ^{2}}{2}f^{\text{ }\prime\prime}\left(  t\right)
dt\right\vert \nonumber\\
&  \leq\frac{1}{\left(  b-a\right)  }\left[  \frac{\left[  x-\left(
a+h\frac{b-a}{2}\right)  \right]  ^{2}}{2}\int\limits_{a}^{x}\left\Vert
f^{\text{ }\prime\prime}\left(  t\right)  \right\Vert dt+\frac{\left[  \left(
b-h\frac{b-a}{2}\right)  -x\right]  ^{2}}{2}\int\limits_{x}^{b}\left\Vert
f^{\text{ }\prime\prime}\left(  t\right)  \right\Vert dt\right] \nonumber\\
&  \leq\frac{1}{\left(  b-a\right)  }\max\left\{  \frac{\left[  x-\left(
a+h\frac{b-a}{2}\right)  \right]  ^{2}}{2},\frac{\left[  \left(  b-h\frac
{b-a}{2}\right)  -x\right]  ^{2}}{2}\right\} \nonumber\\
&  \times\left[  \int\limits_{a}^{x}\left\Vert f^{\text{ }\prime\prime}\left(
t\right)  \right\Vert dt+\int\limits_{x}^{b}\left\Vert f^{\text{ }\prime
\prime}\left(  t\right)  \right\Vert dt\right]  \label{7}%
\end{align}
Now observe that%
\begin{align}
&  \max\left\{  \frac{\left[  x-\left(  a+h\frac{b-a}{2}\right)  \right]
^{2}}{2},\frac{\left[  \left(  b-h\frac{b-a}{2}\right)  -x\right]  ^{2}}%
{2}\right\} \label{8}\\
& \nonumber\\
&  =\frac{1}{2}\left[  \frac{1}{2}\left(  b-a\right)  \left(  1-h\right)
+\left\vert x-\frac{a+b}{2}\right\vert \right]  ^{2}.\nonumber
\end{align}
Using (\ref{8}) in (\ref{7}), we get our required result (\ref{5}).
\end{proof}

\begin{remark}
For $h=0\mathbf{,}\ \ $in (\ref{5}), we obtain Barnett's result (\ref{4}). It
shows that Barnett's result (\ref{4}) is our special case.
\end{remark}

\begin{remark}
For $h=1$ in (\ref{5}, we obtain another result:%
\begin{align}
&  \left\vert
\begin{array}
[c]{c}%
\frac{1}{2}\left(  f\left(  a\right)  +f\left(  b\right)  \right)
-\frac{\left(  b-a\right)  }{8}\left(  f^{\text{ }\prime}\left(  b\right)
-f^{\text{ }\prime}\left(  a\right)  \right)  -\frac{1}{b-a}\int%
\limits_{a}^{b}f(t)dt\\
\end{array}
\right\vert \label{9}\\
&  \leq\frac{1}{2\left(  b-a\right)  }\left[  \left\vert x-\frac{a+b}%
{2}\right\vert \right]  ^{2}\left\Vert f^{\text{ }\prime\prime}\right\Vert
_{1}\nonumber\\
&  \leq\frac{\left(  b-a\right)  }{8}\left\Vert f^{\text{ }\prime\prime
}\right\Vert _{1}.\nonumber
\end{align}
Hence for different values of $h$, we can obtain a variety of results.
\end{remark}

\begin{corollary}
If $f$ \ is as in Theorem $5$, then we have the following perturbed midpoint
inequality:%
\begin{align}
&  \left\vert
\begin{array}
[c]{c}%
\left(  1-h\right)  \text{ }f\left(  \frac{a+b}{2}\right)  +\frac{h}{2}\left(
f\left(  a\right)  +f\left(  b\right)  \right) \\
-\frac{h^{2}\left(  b-a\right)  }{8}\left(  f^{\text{ }\prime}\left(
b\right)  -f^{\text{ }\prime}\left(  a\right)  \right)  -\frac{1}{b-a}%
\int\limits_{a}^{b}f(t)dt\\
\end{array}
\right\vert \label{10}\\
&  \leq\frac{1}{8}\left(  b-a\right)  \left(  1-h\right)  ^{2}\left\Vert
f^{\text{ }\prime\prime}\right\Vert _{1}\nonumber
\end{align}
and for $h=0$, we have%
\begin{equation}
\left\vert \text{ }f\left(  \frac{a+b}{2}\right)  -\frac{1}{b-a}%
\int\limits_{a}^{b}f(t)dt\right\vert \leq\frac{\left(  b-a\right)  }%
{8}\left\Vert f^{\text{ }\prime\prime}\right\Vert _{1}. \label{11}%
\end{equation}

\end{corollary}

\begin{corollary}
Let $f$ \ be as in Theorem $5$, then%
\begin{align}
&  \left\vert
\begin{array}
[c]{c}%
\left(  1-h\right)  \frac{\text{ }f(a)+f(b)}{2}-\left(  1-h\right)  \left(
b-a\right)  \frac{f^{\text{ }\prime}\left(  b\right)  -f^{\text{ }\prime
}\left(  a\right)  }{4}+\frac{h}{2}\left(  f\left(  a\right)  +f\left(
b\right)  \right) \\
-\frac{h^{2}\left(  b-a\right)  }{8}\left(  f^{\text{ }\prime}\left(
b\right)  -f^{\text{ }\prime}\left(  a\right)  \right)  -\frac{1}{b-a}%
\int\limits_{a}^{b}f(t)dt
\end{array}
\right\vert \nonumber\\
&  \leq\frac{\left(  b-a\right)  }{2}\left(  1-\frac{h}{2}\right)
^{2}\left\Vert f^{\text{ }\prime\prime}\right\Vert _{1}\nonumber
\end{align}
holds.
\end{corollary}

\begin{proof}
Put $x=a$ and $x=b$ in (\ref{5}), summing up the obtained inequalities, using
the triangle inequality and dividing by 2, we get the required inequality
(\ref{12}).
\end{proof}

\begin{corollary}
Let $f$ \ be as in Theorem $5$, then we have the perturbed trapezoidal
inequality:%
\begin{align}
&  \left\vert \frac{\text{ }f(a)+f(b)}{2}-\left(  b-a\right)  \frac{f^{\text{
}\prime}\left(  b\right)  -f^{\text{ }\prime}\left(  a\right)  }{4}-\frac
{1}{b-a}\int\limits_{a}^{b}f(t)dt\right\vert \label{13}\\
&  \leq\frac{1}{6}\left(  b-a\right)  ^{2}\left\Vert f^{\text{ }\prime\prime
}\right\Vert _{\infty}.\nonumber
\end{align}

\end{corollary}

\begin{proof}
Put $h=0,$ in (\ref{12}).
\end{proof}

\begin{remark}
The estimation provided by (\ref{12}), is similar to that of the classical
trapezoidal inequality$.$
\end{remark}

\section{\textbf{Applications in Numerical integration}}

Let $I_{n}:$ $a=x_{0}<x_{1}<x_{2}<....<x_{n-1}<x_{n}=b$ be a division of the
interval $[a,b]$ , \ $\xi_{i}\in\left[  x_{i}+\delta\frac{h_{i}}{2}%
,x_{i+1}-\delta\frac{h_{i}}{2}\right]  ,$ $\left(  i=0,1,.....,n-1\right)  $ a
sequence of intermediate points and $h_{i}=x_{i+1}-x_{i}$ , $\left(
i=0,1,.....,n-1\right)  .$ then we have the following quadrature rule:

\begin{theorem}
\textit{Let \ }$\mathit{\ }f:[a,b]\rightarrow%
\mathbb{R}
$ \textit{be a} twice differentiable on $\left(  a,b\right)  $ whose second
derivative and $f^{\text{ }\prime\prime}:\left(  a,b\right)  \rightarrow%
\mathbb{R}
$ belongs to $L_{1}\left(  a,b\right)  $, i.e $\left\Vert f^{\prime\prime
}\right\Vert _{1}:=\int\limits_{a}^{b}\left\Vert f^{\prime\prime}\right\Vert
dt<\infty.$ Then the perturbed Riemann's quadrature formula holds:%
\begin{equation}
\int\limits_{a}^{b}f(t)dt=A\left(  f,f^{%
\acute{}%
},I_{n},\xi,\delta\right)  +R\left(  f,f^{%
\acute{}%
},I_{n},\xi,\delta\right)  , \label{14}%
\end{equation}
where%
\begin{align}
&  A\left(  f,f^{%
\acute{}%
},I_{n},\xi,\delta\right) \label{15}\\
&  =\left(  1-\delta\right)  \underset{i=0}{\overset{n-1}{\sum}}h_{i}f(\xi
_{i})-\left(  1-\delta\right)  \underset{i=0}{\overset{n-1}{\sum}}h_{i}\left(
\xi_{i}-\frac{x_{i}+x_{i+1}}{2}\right)  f^{%
\acute{}%
}(\xi_{i})\nonumber\\
&  +\frac{\delta}{2}\underset{i=0}{\overset{n-1}{\sum}}h_{i}\left(  f\left(
x_{i}\right)  +f\left(  x_{i+1}\right)  \right)  -\frac{\delta^{2}}%
{8}\underset{i=0}{\overset{n-1}{\sum}}h_{i}^{2}\left(  f^{%
\acute{}%
}\left(  x_{i+1}\right)  -f^{%
\acute{}%
}\left(  x_{i}\right)  \right) \nonumber
\end{align}
and the remainder $R\left(  f,f^{%
\acute{}%
},I_{n},\xi,\delta\right)  $ satisfies the estimation:%
\begin{align}
&  \left\vert R\left(  f,f^{%
\acute{}%
},I_{n},\xi,\delta\right)  \right\vert \label{16}\\
&  \leq\frac{1}{2}\underset{i=0}{\overset{n-1}{\sum}}\left[  \frac
{h_{i}\left(  1-\delta\right)  }{2}+\left\vert \xi_{i}-\frac{x_{i}+x_{i+1}}%
{2}\right\vert \right]  ^{2}\left\Vert f^{\text{ }\prime\prime}\right\Vert
_{1}\nonumber\\
&  \leq\left(  1-\frac{\delta}{2}\right)  ^{2}%
\underset{i=0}{\overset{n-1}{\sum}}\frac{h_{i}^{2}}{2}\left\Vert f^{\text{
}\prime\prime}\right\Vert _{1}\nonumber
\end{align}
where $\delta\in\left[  0,1\right]  $ and $x_{i}+\delta\frac{h_{i}}{2}\leq
\xi_{i}\leq x_{i+1}-\delta\frac{h_{i}}{2}.$
\end{theorem}

\begin{proof}
Apply Theorem 5 on the interval $[x_{i},x_{i+1}]$, $\left(
i=0,1,....n-1\right)  ,$ gives%
\begin{align*}
&  \left\vert
\begin{array}
[c]{c}%
\left(  1-\delta\right)  h_{i}f(\xi_{i})-\left(  1-\delta\right)  h_{i}\left(
\xi_{i}-\frac{x_{i}+x_{i+1}}{2}\right)  f^{%
\acute{}%
}(\xi_{i})+\frac{\delta}{2}h_{i}\left(  f\left(  x_{i}\right)  +f\left(
x_{i+1}\right)  \right) \\
\\
-\frac{\delta^{2}}{8}h_{i}^{2}\left(  f^{%
\acute{}%
}\left(  x_{i+1}\right)  -f^{%
\acute{}%
}\left(  x_{i}\right)  \right)  -\int\limits_{x_{i}}^{x_{i+1}}f(t)dt\\
\end{array}
\right\vert \\
&  \leq\frac{1}{2}\left[  \frac{h_{i}\left(  1-\delta\right)  }{2}+\left\vert
\xi_{i}-\frac{x_{i}+x_{i+1}}{2}\right\vert \right]  ^{2}\left\Vert f^{\text{
}\prime\prime}\right\Vert _{1}\\
&  \leq\left(  1-\frac{\delta}{2}\right)  ^{2}\frac{h_{i}^{2}}{2}\left\Vert
f^{\text{ }\prime\prime}\right\Vert _{1}%
\end{align*}
for any choice $\xi$ of the intermediate points.
Summing over $i$ from $0$ to $n-1$ and using the generalized triangular
inequality, we deduce the desired estimation (\ref{16}).
\end{proof}

\begin{corollary}
\bigskip The following perturbed midpoint rule holds:%
\[
\int\limits_{a}^{b}f(x)dx=M\left(  f,f^{%
\acute{}%
},I_{n}\right)  +R_{M}\left(  f,f^{%
\acute{}%
},I_{n}\right)
\]
where%
\begin{equation}
M\left(  f,f^{%
\acute{}%
},I_{n}\right)  =\underset{i=0}{\overset{n-1}{\sum}}h_{i}f\left(  \frac
{x_{i}+x_{i+1}}{2}\right)  \label{17}%
\end{equation}
and the remainder term $R_{M}\left(  f,f^{%
\acute{}%
},I_{n}\right)  $ satisfies the estimation:%
\begin{equation}
\left\vert R_{M}\left(  f,f^{%
\acute{}%
},I_{n}\right)  \right\vert \leq\left\Vert f^{\text{ }\prime\prime}\right\Vert
_{1}\underset{i=0}{\overset{n-1}{\sum}}\frac{h_{i}^{2}}{8}. \label{18}%
\end{equation}

\end{corollary}

\begin{corollary}
The following perturbed trapezoidal rule holds:%
\begin{equation}
\int\limits_{a}^{b}f(x)dx=T\left(  f,f^{%
\acute{}%
},I_{n}\right)  +R_{T}\left(  f,f^{%
\acute{}%
},I_{n}\right)  \label{19}%
\end{equation}
where%
\begin{equation}
T\left(  f,f^{%
\acute{}%
},I_{n}\right)  =\frac{1}{2}\underset{i=0}{\overset{n-1}{\sum}}h_{i}\left(
f\left(  x_{i}\right)  +f\left(  x_{i+1}\right)  \right)  -\frac{1}%
{8}\underset{i=0}{\overset{n-1}{\sum}}h_{i}^{2}\left(  f^{%
\acute{}%
}\left(  x_{i+1}\right)  -f^{%
\acute{}%
}\left(  x_{i}\right)  \right)  \label{20}%
\end{equation}
and the remainder term%
\begin{equation}
\left\vert R_{T}\left(  f,f^{%
\acute{}%
},I_{n}\right)  \right\vert \leq\underset{i=0}{\overset{n-1}{\sum}}\frac
{h_{i}^{2}}{8}\left\Vert f^{\text{ }\prime\prime}\right\Vert _{1}. \label{21}%
\end{equation}

\end{corollary}

\begin{remark}
Note that the above mentioned perturbed midpoint formula (\ref{17}) and
perturbed trapezoid formula (\ref{20}) can give better approximations of the
integral $\int\limits_{a}^{b}f(x)dx$ for general classes of mappings.
\end{remark}

\section{ Special Means}

We may now apply inequality (\ref{5})$,$ to deduce some inequalities for
special means \cite{15} at [P. 1896] by the use of particular mappings as follows:

\begin{remark}
Consider mapping $f:\left(  0,\infty\right)  \rightarrow%
\mathbb{R}
,$ $f(x)=x^{r}\ ,r\in%
\mathbb{R}
\backslash\left\{  -1,0\right\}  ,$ then we have for \ $0<a<b,$%
\[
\frac{1}{b-a}\int\limits_{a}^{b}f(t)dt=L_{r}^{r}\left(  a,b\right)
\]
Using the inequality (\ref{5}) we get:%
\begin{align}
&  \left\vert
\begin{array}
[c]{c}%
\left(  1-h\right)  x^{r}-\left(  1-h\right)  \left(  x-A\right)  rx^{r-1}\\
+\frac{h}{2}\left(  a^{r}+b^{r}\right)  -\frac{h^{2}\left(  b-a\right)  r}%
{8}\left[  b^{r-1}-a^{r-1}\right]  -L_{r}^{r}\left(  a,b\right)
\end{array}
\right\vert \label{22}\\
& \nonumber\\
&  \leq\frac{1}{2}\left[  \frac{1}{2}\left(  b-a\right)  \left(  1-h\right)
+\left\vert x-A\right\vert \right]  ^{2}\left\vert r\left(  r-1\right)
L_{r-1}^{r-1}\left(  a,b\right)  \right\vert .\nonumber
\end{align}
If in (\ref{22}), we choose $x=A,$ we get%
\begin{align}
&  \left\vert \left(  1-h\right)  A^{r}+\frac{h}{2}\left(  a^{r}+b^{r}\right)
-\frac{h^{2}\left(  b-a\right)  r}{8}\left[  b^{r-1}-a^{r-1}\right]
-L_{r}^{r}\left(  a,b\right)  \right\vert \label{22-1}\\
&  \leq\frac{1}{8}\left[  \left(  b-a\right)  \left(  1-h\right)  \right]
^{2}\left\vert r\left(  r-1\right)  L_{r-1}^{r-1}\left(  a,b\right)
\right\vert \nonumber
\end{align}
\bigskip also choosing $h=0$ in (\ref{22-1}), we get%
\begin{equation}
\left\vert A^{r}-L_{r}^{r}\left(  a,b\right)  \right\vert \leq\frac{1}%
{8}\left(  b-a\right)  ^{2}\left\vert r\left(  r-1\right)  L_{r-1}%
^{r-1}\left(  a,b\right)  \right\vert . \label{24}%
\end{equation}

\end{remark}

\begin{remark}
Consider the mapping $f:\left(  0,\infty\right)  \rightarrow%
\mathbb{R}
,$ \ $f(x)=\frac{1}{x}\ .$

Then we have for \ $0<a<b,$%
\[
\frac{1}{b-a}\int\limits_{a}^{b}f(t)dt=L^{-1}\left(  a,b\right)
\]
Using the inequality (\ref{5}) we get:%
\begin{align}
&  \left\vert \left(  1-h\right)  \text{ }\frac{1}{x}+\left(  1-h\right)
\left(  x-A\right)  \frac{1}{x^{2}}+\frac{h}{H}\ -\frac{h^{2}\left(
b-a\right)  }{8}\left(  \frac{b^{2}-a^{2}}{a^{2}b^{2}}\right)  -L^{-1}\left(
a,b\right)  \right\vert \label{24-1}\\
&  \leq\left[  \frac{1}{2}\left(  b-a\right)  \left(  1-h\right)  +\left\vert
x-A\right\vert \right]  ^{2}L_{-3}^{-3}\left(  a,b\right)  .\nonumber
\end{align}
If in (\ref{24-1}), we choose $x=A,$ we get%
\begin{align}
&  \left\vert \left(  1-h\right)  \text{ }\frac{1}{A}+\frac{h}{H}%
\ -\frac{h^{2}\left(  b-a\right)  }{8}\left(  \frac{b^{2}-a^{2}}{a^{2}b^{2}%
}\right)  -L^{-1}\left(  a,b\right)  \right\vert \label{24-2}\\
&  \leq\frac{1}{4}\left(  b-a\right)  ^{2}\left(  1-h\right)  ^{2}L_{-3}%
^{-3}\left(  a,b\right) \nonumber
\end{align}
also choosing $h=0$ in (\ref{24-2}), we get%
\begin{equation}
\left\vert \text{ }A^{-1}\ -L^{-1}\left(  a,b\right)  \right\vert \leq\frac
{1}{4}\left(  b-a\right)  ^{2}L_{-3}^{-3}\left(  a,b\right)  . \label{27}%
\end{equation}

\end{remark}

\begin{remark}
Let us consider\ the mapping\ \ $f(x)=\ln x$ $,$\ $x\in\left[  a,b\right]  $
$\subset\left(  0,\infty\right)  .$

Then we have:%
\[
\frac{1}{b-a}\int\limits_{a}^{b}f(t)dt=\ln I\left(  a,b\right)
\]
Using the inequality (\ref{5}) we get:%
\begin{align}
&  \left\vert
\begin{array}
[c]{c}%
\left(  1-h\right)  \text{ }\ln x-\left(  1-h\right)  \left(  x-A\right)
\frac{1}{x}+\frac{h}{2}\left(  \ln a+\ln b\right) \\
+\frac{h^{2}\left(  b-a\right)  ^{2}}{8ab}-\ln I\left(  a,b\right)  .
\end{array}
\right\vert \label{28}\\
& \nonumber\\
&  \leq\frac{1}{2}\left[  \frac{1}{2}\left(  b-a\right)  \left(  1-h\right)
+\left\vert x-A\right\vert \right]  ^{2}L_{-2}^{-2}\left(  a,b\right)
.\nonumber
\end{align}
If in (\ref{28}), we choose $x=A,$ we get%
\begin{align}
&  \left\vert \left(  1-h\right)  \text{ }\ln A-\frac{h}{2}\left(  \ln a+\ln
b\right)  +\frac{h^{2}\left(  b-a\right)  ^{2}}{8ab}-\ln I\left(  a,b\right)
\right\vert \label{28-1}\\
&  \leq\frac{1}{8}\left(  b-a\right)  ^{2}\left(  1-h\right)  ^{2}L_{-2}%
^{-2}\left(  a,b\right) \nonumber
\end{align}
also choosing $h=0$ in (\ref{28-1}), we get%
\begin{equation}
\left\vert \frac{A}{I}\right\vert \leq\exp\frac{\left(  b-a\right)  ^{2}}%
{8}L_{-2}^{-2}\left(  a,b\right)  . \label{30}%
\end{equation}

\end{remark}

\section{Conclusion}

We established generalized Ostrowski type inequality for bounded
differentiable mappings which generalizes the previous inequalities developed
and discussed in \cite{6},\cite{7},\cite{10}- \cite{12} . Perturbed midpoint
and trapezoid inequalities are obtained. Some closely new results are also
given.These generalized inequalities add up to the literature in the sense
that they have immediate applications in Numerical Integration and Special
Means. These generalized inequalities will also be useful for the researchers
working in the field of the approximation theory, applied mathematics,
probability theory, stochastic and numerical analysis to solve their problems
in engineering and in practical life.

\end{document}